\documentclass[12pt]{amsart}
\usepackage{mathrsfs}      
\usepackage{txfonts}
\usepackage{amssymb}
\usepackage{eucal}
\usepackage{graphicx}
\usepackage{amsmath}
\usepackage{pmat}
\usepackage{amscd}
\usepackage[all]{xy}           
\usepackage[active]{srcltx} 

\usepackage{amsfonts,latexsym}
\usepackage{xspace}
\usepackage{epsfig}
\usepackage{float}
\usepackage{color}
\usepackage{fancybox}
\usepackage{colordvi}
\usepackage{multicol}
\usepackage{colordvi}
\usepackage[colorlinks,final,backref=page,hyperindex,hypertex]{hyperref}
\newtheorem{theorem}{Theorem}[section]
\newtheorem{prop}[theorem]{Proposition}
\newtheorem{defn}[theorem]{Definition}
\newtheorem{lemma}[theorem]{Lemma}
\newtheorem{coro}[theorem]{Corollary}
\newtheorem{prop-def}{Proposition-Definition}[section]
\newtheorem{coro-def}{Corollary-Definition}[section]

\newtheorem{remark}{Remark}[section]

\newtheorem{exam}{Example}[section]

\newcommand{\nc}{\newcommand}
\nc{\tred}[1]{\textcolor{red}{#1}}
\nc{\tblue}[1]{\textcolor{blue}{#1}}
\nc{\tgreen}[1]{\textcolor{green}{#1}}
\nc{\tpurple}[1]{\textcolor{purple}{#1}}
\nc{\btred}[1]{\textcolor{red}{\bf #1}}
\nc{\btblue}[1]{\textcolor{blue}{\bf #1}}
\nc{\btgreen}[1]{\textcolor{green}{\bf #1}}
\nc{\btpurple}[1]{\textcolor{purple}{\bf #1}}

\newcommand{\efootnote}[1]{}

\renewcommand{\textbf}[1]{}

\newcommand{\delete}[1]{}

\delete{
\nc{\mlabel}[1]{\label{#1}}  
\nc{\mcite}[1]{\cite{#1}}  
\nc{\mref}[1]{\ref{#1}}  
\nc{\mcite}[1]{\cite{#1}{{\bf{{\ }(#1)}}}}  
}

\nc{\mlabel}[1]{\label{#1}  
{\hfill \hspace{1cm}{\bf{{\ }\hfill(#1)}}}}
\nc{\mcite}[1]{\cite{#1}{{\bf{{\ }(#1)}}}}  
\nc{\mref}[1]{\ref{#1}{{\bf{{\ }(#1)}}}}  

\nc{\mbibitem}[1]{\bibitem[\bf #1]{#1}} 
\nc{\mkeep}[1]{\marginpar{{\bf #1}}} 

\nc{\lie}{\mathfrak{g}}

\nc{\bra}{[\cdot,\cdot]}

\nc{\gl}{\mathfrak{g}\mathfrak{l}}

\nc{\Mnk}{\bf M^{n}(k)}

\nc{\E}{{\rm End_{\bfk}}}

\nc{\tr}{\rm Tr}
\nc{\rank}{\rm rank}
\nc{\spa}{\rm span}
\nc{\ad}{{\rm ad}}
\nc{\dep}{{\rm dep}}
\nc{\diag}{{\rm diag}}

\nc{\Id}{{\rm Id}}

\nc{\Ixq}{I_{P}}
\nc{\bin}[2]{ (_{\stackrel{\scs{#1}}{\scs{#2}}})}  
\nc{\binc}[2]{ \left (\!\! \begin{array}{c} \scs{#1}\\
    \scs{#2} \end{array}\!\! \right )}  
\nc{\bincc}[2]{  \left ( {\scs{#1} \atop
    \vspace{-1cm}\scs{#2}} \right )}  
\nc{\bs}{\bar{S}} \nc{\cosum}{\sqsubset} \nc{\la}{\longrightarrow}
\nc{\rar}{\rightarrow} \nc{\dar}{\downarrow} \nc{\dprod}{**}
\nc{\dap}[1]{\downarrow \rlap{$\scriptstyle{#1}$}}
\nc{\md}{\mathrm{dth}} \nc{\uap}[1]{\uparrow
\rlap{$\scriptstyle{#1}$}} \nc{\defeq}{\stackrel{\rm def}{=}}
\nc{\disp}[1]{\displaystyle{#1}} \nc{\dotcup}{\
\displaystyle{\bigcup^\bullet}\ } \nc{\gzeta}{\bar{\zeta}}
\nc{\hcm}{\ \hat{,}\ } \nc{\hts}{\hat{\otimes}}
\nc{\barot}{{\otimes}} \nc{\free}[1]{\bar{#1}}
\nc{\uni}[1]{\tilde{#1}} \nc{\hcirc}{\hat{\circ}} \nc{\lleft}{[}
\nc{\lright}{]} \nc{\lc}{\lfloor} \nc{\rc}{\rfloor}
\nc{\curlyl}{\left \{ \begin{array}{c} {} \\ {} \end{array}
    \right .  \!\!\!\!\!\!\!}
\nc{\curlyr}{ \!\!\!\!\!\!\!
    \left . \begin{array}{c} {} \\ {} \end{array}
    \right \} }
\nc{\longmid}{\left | \begin{array}{c} {} \\ {} \end{array}
    \right . \!\!\!\!\!\!\!}
\nc{\onetree}{\bullet} \nc{\ora}[1]{\stackrel{#1}{\rar}}
\nc{\ola}[1]{\stackrel{#1}{\la}}
\nc{\ot}{\otimes} \nc{\mot}{{{\boxtimes\,}}}
\nc{\otm}{\overline{\boxtimes}} \nc{\sprod}{\bullet}
\nc{\scs}[1]{\scriptstyle{#1}} \nc{\mrm}[1]{{\rm #1}}
\nc{\margin}[1]{\marginpar{\rm #1}}   
\nc{\dirlim}{\displaystyle{\lim_{\longrightarrow}}\,}
\nc{\invlim}{\displaystyle{\lim_{\longleftarrow}}\,}
\nc{\mvp}{\vspace{0.3cm}} \nc{\tk}{^{(k)}} \nc{\tp}{^\prime}
\nc{\ttp}{^{\prime\prime}} \nc{\svp}{\vspace{2cm}}
\nc{\vp}{\vspace{8cm}} \nc{\proofbegin}{\noindent{\bf Proof: }}
\nc{\proofend}{$\blacksquare$ \vspace{0.3cm}}
\nc{\modg}[1]{\!<\!\!{#1}\!\!>}
\nc{\intg}[1]{F_C(#1)} \nc{\lmodg}{\!
<\!\!} \nc{\rmodg}{\!\!>\!}
\nc{\cpi}{\widehat{\Pi}}
\nc{\sha}{{\mbox{\cyr X}}}  
\nc{\shap}{{\mbox{\cyrs X}}} 
\nc{\shpr}{\diamond}    
\nc{\shp}{\ast} \nc{\shplus}{\shpr^+}
\nc{\shprc}{\shpr_c}    
\nc{\msh}{\ast} \nc{\zprod}{m_0} \nc{\oprod}{m_1}
\nc{\vep}{\varepsilon} \nc{\labs}{\mid\!} \nc{\rabs}{\!\mid}

\nc{\mmbox}[1]{\mbox{\ #1\ }} \nc{\fp}{\mrm{FP}}
\nc{\rchar}{\mrm{char}} \nc{\End}{\mrm{End}} \nc{\Fil}{\mrm{Fil}}
\nc{\Mor}{Mor\xspace} \nc{\gmzvs}{gMZV\xspace}
\nc{\gmzv}{gMZV\xspace} \nc{\mzv}{MZV\xspace}
\nc{\mzvs}{MZVs\xspace} \nc{\Hom}{\mrm{Hom}} \nc{\id}{\mrm{id}}
\nc{\im}{\mrm{im}} \nc{\incl}{\mrm{incl}} \nc{\map}{\mrm{Map}}
\nc{\mchar}{\rm char} \nc{\nz}{\rm NZ} \nc{\supp}{\mathrm Supp}

\nc{\Alg}{\mathbf{Alg}} \nc{\Bax}{\mathbf{Bax}} \nc{\bff}{\mathbf f}
\nc{\bfk}{{\bf k}} \nc{\bfone}{{\bf 1}} \nc{\bfx}{\mathbf x}
\nc{\bfy}{\mathbf y}
\nc{\base}[1]{\bfone^{\otimes ({#1}+1)}} 
\nc{\Cat}{\mathbf{Cat}}

\nc{\detail}{\marginpar{\bf More detail}
    \noindent{\bf Need more detail!}
    \svp}
\nc{\Int}{\mathbf{Int}} \nc{\Mon}{\mathbf{Mon}}
\nc{\rbtm}{{shuffle }} \nc{\rbto}{{Rota-Baxter }}
\nc{\remarks}{\noindent{\bf Remarks: }} \nc{\Rings}{\mathbf{Rings}}
\nc{\Sets}{\mathbf{Sets}} \nc{\wtot}{\widetilde{\odot}}
\nc{\wast}{\widetilde{\ast}} \nc{\bodot}{\bar{\odot}}
\nc{\bast}{\bar{\ast}} \nc{\hodot}[1]{\odot^{#1}}
\nc{\hast}[1]{\ast^{#1}} \nc{\mal}{\mathcal{O}}
\nc{\tet}{\tilde{\ast}} \nc{\teot}{\tilde{\odot}}
\nc{\oex}{\overline{x}} \nc{\oey}{\overline{y}}
\nc{\oez}{\overline{z}} \nc{\oef}{\overline{f}}
\nc{\oea}{\overline{a}} \nc{\oeb}{\overline{b}}
\nc{\weast}[1]{\widetilde{\ast}^{#1}}
\nc{\weodot}[1]{\widetilde{\odot}^{#1}} \nc{\hstar}[1]{\star^{#1}}
\nc{\lae}{\langle} \nc{\rae}{\rangle}
\nc{\lf}{\lfloor}\nc{\rf}{\rfloor}

\nc{\cala}{{\mathcal A}} \nc{\calb}{{\mathcal B}}
\nc{\calc}{{\mathcal C}}
\nc{\cald}{{\mathcal D}} \nc{\cale}{{\mathcal E}}
\nc{\calf}{{\mathcal F}} \nc{\calg}{{\mathcal G}}
\nc{\calh}{{\mathcal H}} \nc{\cali}{{\mathcal I}}
\nc{\calj}{{\mathcal J}} \nc{\calk}{{\mathcal K}} \nc{\call}{{\mathcal L}} \nc{\calm}{{\mathcal M}}
\nc{\caln}{{\mathcal N}} \nc{\calo}{{\mathcal O}}
\nc{\calp}{{\mathcal P}} \nc{\calr}{{\mathcal R}}
\nc{\cals}{{\mathcal S}} \nc{\calt}{{\mathcal T}}
\nc{\calu}{{\mathcal U}} \nc{\calv}{{\mathcal V}}
\nc{\calw}{{\mathcal W}}
\nc{\calx}{{\mathcal X}} \nc{\CA}{\mathcal{A}}

\nc{\fraka}{{\mathfrak a}} \nc{\frakA}{{\mathfrak A}}
\nc{\frakb}{{\mathfrak b}} \nc{\frakB}{{\mathfrak B}}
\nc{\frakD}{{\mathfrak D}} \nc{\frakg}{{\mathfrak g}}
\nc{\frakH}{{\mathfrak H}} \nc{\frakL}{{\mathfrak L}}
\nc{\frakM}{{\mathfrak M}} \nc{\bfrakM}{\overline{\frakM}}
\nc{\frakm}{{\mathfrak m}} \nc{\frakP}{{\mathfrak P}}
\nc{\frakN}{{\mathfrak N}} \nc{\frakp}{{\mathfrak p}}
\nc{\frakS}{{\mathfrak S}} \nc{\frakW}{{\mathfrak W}}
\nc{\frakw}{{\mathfrak w}}

\font\cyr=wncyr10 \font\cyrs=wncyr7
\nc{\li}[1]{\textcolor{red}{LG:#1}}
\nc{\jun}[1]{\textcolor{blue}{#1}}

\RequirePackage[normalem]{ulem} 
\RequirePackage{color}\definecolor{RED}{rgb}{1,0,0}\definecolor{BLUE}{rgb}{0,0,1} 
\providecommand{\DIFaddend}{} 

\begin{document}

\title{Representations of Polynomial Rota--Baxter Algebras}
%
\author{Li Qiao}
\address{School of Mathematics and Statistics,
Lanzhou University,
Lanzhou 730000, Gansu, China}
\email{ liqiaomail@126.com}

\author{Jun Pei}
\address{School of Mathematics and Statistics, Southwest University, Chongqing 400715, China}
         \email{peitsun@swu.edu.cn}

\date{\today}

\begin{abstract}
A Rota--Baxter operator is an algebraic abstraction of integration, which is the typical example of a weight zero Rota-Baxter operator. We show that studying the modules over the polynomial Rota--Baxter algebra $(\bfk[x],P)$ is equivalent to studying the modules over the Jordan plane, and we generalize the direct decomposability results for the $(\bfk[x],P)$-modules in \cite{Iy} from algebraically closed fields of characteristic zero to fields of characteristic zero. Furthermore, we provide a classification of Rota--Baxter modules up to isomorphism based on indecomposable $\bfk[x]$-modules.
\end{abstract}

\subjclass[2010]{16W99, 16G99, 45N05, 12H20}

\keywords{integration, Jordan plane, module, polynomial Rota--Baxter algebra, Rota--Baxter operator}

\maketitle

\tableofcontents

\setcounter{section}{0}

\section{Introduction}
The notion of a Rota--Baxter algebra was introduced by Baxter \cite{Baxter}
based on his probability study to understand Spitzer's identity in
fluctuation, where it is a $\bfk$-algebra $R$ with a $\bfk$-linear operator $P$ that satisfies
\begin{equation}\label{equ:rbequ}
P(u)P(v) = P(uP(v))+ P(P(u)v)  + \lambda P(uv), \quad \forall u, v \in R, \lambda \in \bfk.
\end{equation}
The operator $P$ is called a Rota--Baxter operator of weight $\lambda$ on $R$. Rota--Baxter operators are rooted deeply in analysis. A Rota--Baxter algebra can be regarded as the integral analogue of a differential algebra. More fundamentally, in the case where $\lambda = 0$, the Rota--Baxter equation Eq.(\ref{equ:rbequ}) is an algebraic abstraction
of the integration by parts formula for calculus:
$$
FG\big|_{0}^{x} = \int_{0}^{x}F'G+\int_{0}^{x}FG'.
$$

Throughout the 1960s, Rota--Baxter operators were studied widely by researchers such as Atkinson \cite{Atk}. In the 1960s and 1970s, Rota \cite{Rota1,Rota2} and Cartier \cite{C} developed Rota--Baxter operators related to combinatorics. In the late 1990s, the operator appeared again as a fundamental algebraic structure in a study by Connes and Kreimer \cite{CK1} regarding the renormalization of quantum field theory. Subsequently, important theoretical developments and applications of Rota--Baxter algebras have occurred in areas such as mathematical physics, operads, number theory, and combinatorics. Readers may refer to the book by Guo \cite{Guo} for further details of their development.

At present, much of the research into certain algebraic objects in mathematics
involves the theoretical study of their representations. However, the representation theory of Rota--Baxter algebras is still in the early stage of development. The concept and some basic properties of representations of Rota--Baxter algebras were introduced by Guo et al. \cite{GQ}. Similar to the representations of differential algebras, which can be considered as a vector bundle with a connection, a representation of a Rota--Baxter algebra
comprises a module of the algebra together with an integral operator $p$. The
representations and regular-singular decomposition of Laurent series Rota--Baxter algebra were studied by Lin and Qiao \cite{LQ}, and they are important for the renormalization of quantum field theory.

In this study, we revisit the analytical origin of Rota--Baxter operators in order to study the representation theory of Rota--Baxter algebras of weight zero, where $\bfk[x]$ plays a central role in both the analysis and algebra. If $P$ is the standard integral operator on $\bfk[x]$, then $(\bfk[x],P)$ is a Rota--Baxter algebra of weight zero. This provides an ideal testing ground for understanding the interaction between analytically defined Rota--Baxter operators and algebraically defined Rota--Baxter operators.

\subsection{Overview of our results and methods}
The main aim of this study is to investigate finite dimensional $(\bfk[x],P)$-modules. It is well known that the category of modules over a differential algebra is equivalent to the category of modules over its corresponding algebra of differential operators. Thus, our first step involves transforming the problem concerning representations of $(\bfk[x],P)$ into the problem of representing a certain type of algebra in the usual sense. This problem is considered in Subsection \ref{Jordan}, where we show that studying the $(\bfk[x],P)$-modules is equivalent to studying the modules of $\mathcal{J}=\bfk\langle x,y \rangle/(xy-yx-y^{2})$, which is a special example of Ore extensions.

The Ore extension $R[x;\sigma,\delta]$ is the noncommutative ring obtained by giving $R[x]$ a new multiplication subject to the identity $xr=\sigma(r)x+\delta(r)$, where $R$ is a $\bfk$-algebra, $\sigma \in {\rm End}_{\bfk}(R)$, and $\delta$ is a $\sigma$-derivation of $R$. The Ore extensions where $\bfk[y]$ is the underlying ring belong to three specific types, i.e., quantum planes, quantum Weyl algebras, or algebras comprising $\mathcal{A}_{h}=\bfk\langle x,y \rangle/(xy-yx-h)$, where $h \in \bfk[y]$ (see \mbox{\cite{Avv}}). Thus, studying \DIFaddend $(\bfk[x],P)$-modules is equivalent to studying modules of the Jordan plane $\mathcal{J}=\mathcal{A}_{y^{2}}$, and thus the known results for $\mathcal{A}_{y^{2}}$-modules can immediately give their corresponding descriptions of $(\bfk[x],P)$-modules.

The Jordan plane $\mathcal{A}_{y^{2}}$ arises in noncommutative algebraic geometry \cite{Ars,Sz} and it exhibits many interesting features such as being Artin--Schelter regular of dimension $2$. In several studies \cite{Sh1,Sh2,Sh3}, Shirikov extensively investigated the automorphisms, derivations, prime ideals, and modules of the algebra $\mathcal{A}_{y^{2}}$. These investigations were extended by Iyudu \cite{Iy} to include results based on various finite-dimensional modules of $A_{y^{2}}$ over algebraically closed fields of characteristic zero. For finite dimensional modules, Iyudu gave a complete description of the irreducible $\mathcal{A}_{y^{2}}$-module and a necessary condition for indecomposable $\mathcal{A}_{y^{2}}$-modules, which implies the direct decomposition of $\mathcal{A}_{y^{2}}$-modules. Recent studies by Benkart et al. \cite{BLO1,BLO2,BLO3} investigated the family of algebras comprising $\mathcal{A}_{h}$. In \cite{BLO1,BLO3}, they studied the structure of $\mathcal{A}_{h}$ and determined its automorphism group and derivations. In \cite{BLO2}, they gave an example of an indecomposable module of dimension $n+1$ of $\mathcal{A}_{h}$ and completely characterized the irreducible modules of $\mathcal{A}_{h}$ for arbitrary field, thereby generalizing the results obtained for irreducible modules by \mbox{\cite{Iy}}. Thus, we obtain the description of irreducible $(\bfk[x],P)$-modules by interpreting the descriptions of the irreducible $\mathcal{A}_{y^{2}}$-module in terms of the $(\bfk[x],P)$-module language.

In the remainder of Section \ref{sec:general}, by determining the $(\bfk[x],P)$-module structures on a given $\bfk[x]$-module, we generalize the results of regarding the direct decomposability of $(\bfk[x],P)$-modules given by \cite{Iy} from algebraically closed fields of characteristic zero to fields of characteristic zero. This result allows us to reduce the problem of finding the structure of $(\bfk[x],P)$-module $(M,p)$ to the problem of finding the structure of its submodule.

In Section \ref{sec:classify}, as a special case, when $\bfk$ is an algebraically closed field of characteristic zero, we completely and explicitly characterize the Rota--Baxter module structures on the indecomposable $\bfk[x]$-modules.

\section{Representations of $(\bfk[x],P)$}\label{sec:general}
In this section, after some basic definitions, we show that the representations of $(\bfk[x],P)$ are equivalent to representations of the Jordan plane. We then provide a direct decomposition of
$(\bfk[x],P)$-modules over a field of characteristic zero.

\subsection{Representations of $(\bfk[x],P)$ and the Jordan plane} \label{Jordan}
\begin{defn}
{\rm Let $\bfk$ be a field and $(R,P)$ a Rota--Baxter $\bfk$-algebra of weight zero. A (left) Rota--Baxter module over
$(R,P)$ or simply an (left) $(R,P)$-module is a pair $(M,p)$, where $M$ is an
 $R$-module and $p: M \longrightarrow M$ is a $\bfk$-linear map that satisfies
\begin{equation}\label{eq:meq}
P(r)p(m) = p(P(r)m +rp(m)), \quad \forall ~r\in R,  \forall m\in M.
\end{equation}
}
\end{defn}
If we let $(M,p)$ be an $(R,P)$-module, then $M$ is a $\bfk$-vector space.
If $\dim_{\bfk}M < +\infty$, then $(M,p)$ is called a finite dimensional
$(R,P)$-module. In the following, all $(R,P)$-modules are
assumed to be finite dimensional.

Basic concepts regarding an $R$-module can be defined in a similar manner for $(R,P)$-modules.
In particular, an $(R,P)$-module homomorphism
$\phi : (M,p) \longrightarrow (N,q)$ is an $R$-module homomorphism
$\phi: M \longrightarrow M$ that satisfies
$$
\phi \circ p = q \circ \phi.
$$
Furthermore, $(M,p)$ is isomorphic to $(N,q)$ if the homomorphism $\phi$ is bijective.
It is simple to check that $\bigoplus_{i=1}^{n}(M_{i},p_{i}) = (\bigoplus_{i=1}^{n} M_{i},\displaystyle  \sum_{i=1}^{n} p_{i}),$ where $\displaystyle \sum_{i=1}^{n}p_{i}$ is defined by
$$
\left(\sum_{i=1}^{n}p_{i} \right) (u_{1},\cdots ,u_{n}) = \sum_{i=1}^{n}p_{i}(u_{i}),
$$
is still an $(R,P)$-module and it is called the direct sum of $(R,P)$-modules $(M_{1},p_{1}), \cdots, (M_{n},p_{n})$.

\begin{remark}
Recall that $(\bfk[x],P)$ is a Rota--Baxter algebra, where $P: \bfk[x] \longrightarrow \bfk[x]$ is defined by $P(x^{m}) = \frac{1}{m+1}x^{m+1}, m \in \mathbb{N}$ (i.e., the standard integral operator). Thus, $P$ is only well defined over a field of characteristic zero. In the following, we let $\bfk$ be a field of characteristic zero, unless stated otherwise.
\end{remark}

Let $(\bfk[x],P)$ be the polynomial Rota--Baxter algebra and $\bfk \langle x,y\rangle$ the noncommutative polynomial algebra with variables $x$ and $y$. Let $\Ixq$ be the ideal of $\bfk \langle x,y\rangle$ generated by the set
\begin{equation}
\mathcal{X}_{0}=\{P(f)y-yP(f)-yfy~|~f\in \bfk[x]\},
\end{equation}
and $\mathcal{J} = \bfk \langle x,y\rangle/\Ixq$.

For any $m\in \mathbb{N}$, we have
$$
P(x^{m})y-yP(x^{m})-yx^{m}y=\frac{x^{m+1}}{m+1}y-\frac{1}{m+1}yx^{m+1}-yx^{m}y.
$$
The integral operator $P$ is $\bfk$-linear, so the set
$$
\mathcal{X}_{1}=\{x^{m+1}y-yx^{m+1}-(m+1)yx^{m}y~|~m\in \mathbb{N}\}
$$
also generates $\Ixq$.

Next, we establish the relationship between $\mathcal{J}$-modules and $(\bfk[x],P)$-modules. Recall that a $(\bfk[x],P)$-module is a pair $(M,p)$, where $M$ is a $\bfk[x]$-module and $p \in {\rm End}_{\bfk}(M)$ such that
\begin{equation}
P(f)p(v)=p(P(f)v+fp(v)), \quad \forall f \in \bfk[x],~~ \forall v \in M.
\end{equation}

\begin{prop}\label{prop:modulerel}
Let $M$ be a $\mathcal{J}$-module. Define a $\bfk$-linear map $p$ on $M$ by
$$
p(v) = yv, \quad \forall v \in M.
$$
Then, $(M,p)$ is a $(\bfk[x],P)$-module. Conversely, if $(M,p)$ is a $(\bfk[x],P)$-module and we define
$$
yv = p(v), \quad \forall v \in M,
$$
then $M$ is a $\mathcal{J}$-module.
\end{prop}
\begin{proof}
We note that the equations
$$
P(x^{m})y-yP(x^{m})-yx^{m}y =0, \quad m \in \mathbb{N},
$$
hold in $\mathcal{J}$. Thus, for any $v \in M$, we have
$$
(P(x^{m})y-yP(x^{m})-yx^{m}y)(v)=0,
$$
i.e.,
$$
P(x^{m})p(v)=p(P(x^{m})v+x^{m}p(v)).
$$
Hence, $(M,p)$ is a $(\bfk[x],P)$-module.

Conversely, $M$ is now a $\bfk\langle x,y \rangle$-module. $(M,p)$ is a $(\bfk[x],P)$-module, so we have
$$
\Ixq \subseteq {\rm ann}M = \{F \in \bfk\langle x,y\rangle~|~Fv=0, \mbox{for all}~v \in M\}.
$$
Thus, $M$ is a $\mathcal{J}$-module.
\end{proof}

Due to Proposition \ref{prop:modulerel}, the study of $(\bfk[x],P)$-modules becomes the study of $\mathcal{J}$-modules in the usual sense.

\begin{lemma}\label{lem:simplify}
If we let $\mathcal{X}=\{xy-yx-y^{2}\}$, then $\Ixq$ is generated by $\mathcal{X}$, i.e., $\Ixq := \Id(\mathcal{X})$.
\end{lemma}
\begin{proof}
Obviously, we have $\Id(\mathcal{X}) \subseteq \Ixq$. Conversely, we need to show that $\Ixq \subseteq \Id(\mathcal{X})$. We use induction on $m$ to show that $x^{m+1}y-yx^{m+1}-(m+1)yx^{m}y$ is in $\Id(\mathcal{X})$, but there is no need to consider $m=0$. Assume that for any $m-1 \geq 0$, $x^{m}y-yx^{m}-myx^{m-1}y \in \Id(\mathcal{X})$. For $m \geq 1$, we have
\begin{eqnarray*}
&&x^{m+1}y-yx^{m+1}-(m+1)yx^{m}y \\
&&= x^{m+1}y-x^{m}yx-x^{m}y^{2} + x^{m}yx+x^{m}y^{2}- yx^{m+1}-(m+1)yx^{m}y\\
&&=x^{m}(xy-yx-y^{2})+(x^{m}y-yx^{m}-myx^{m-1}y)x+myx^{m-1}y x+x^{m}y^{2}-(m+1)yx^{m}y,
\end{eqnarray*}
and
\begin{eqnarray*}
&& myx^{m-1}y x+x^{m}y^{2}-(m+1)yx^{m}y \\
 &&=myx^{m-1}(yx - xy+y^{2})+(x^{m}y-yx^{m}-myx^{m-1}y)y.
\end{eqnarray*}
Then,
\begin{eqnarray*}
&&x^{m+1}y-yx^{m+1}-(m+1)yx^{m}y \\
&&=(x^{m}-myx^{m-1})(yx - xy+y^{2})+(x^{m}y-yx^{m}-myx^{m-1}y)(x+y).
\end{eqnarray*}
By the induction hypothesis, we obtain $x^{m+1}y-yx^{m+1}-(m+1)yx^{m}y \in \Id(\mathcal{X})$, i.e., $\mathcal{X}_{1} \subset \Id(\mathcal{X})$ and $\Ixq=\Id(\mathcal{X}_{1}) \subseteq \Id(\mathcal{X})$. Therefore,
$\Ixq=\Id(\mathcal{X})$.
\end{proof}

\begin{lemma}\label{lem:ker}
For any $m\geq 1$, the equation $y^{m}x = xy^{m}-my^{m+1}$ holds in $\mathcal{J}$.
\end{lemma}
\begin{proof}
We prove the statement by induction on $m$. $xy-yx=y^{2}$ holds in $\mathcal{J}$, so the case where $m=1$ is obvious. Assume that the equation holds for $m \geq1 $. We consider the case where $m+1$:
 \begin{eqnarray*}
 xy^{m+1}-y^{m+1}x &=& xy^{m+1}-yx y^{m}+yxy^{m}-y^{m+1}x\\
&=& (xy-yx)y^{m}+y(x y^{m}-y^{m}x)\\
&=& y^{2} y^{m} +y(my^{m+1}) = (m+1)y^{m+2}.
\end{eqnarray*}
This completes the induction.
\end{proof}

\begin{coro}\label{thm:sim}
Let $M$ be a $\bfk[x]$-module and $p \in {\rm End}_{\bfk}(M)$. Then, $(M,p)$ is a $(\bfk[x],P)$-module if and only if
\begin{equation}\label{equ:main}
xp-px=p^{2}.
\end{equation}
\end{coro}

\begin{coro}\label{thm:nil}
Let $(M,p)$ be a $(\bfk[x],P)$-module. Then, $p:M \longrightarrow M$ is nilpotent.
\end{coro}

\begin{coro}\label{coro:onedim}
If $(M,p)$ is a $1$-dimensional $(\bfk[x],P)$-module, then $p=0$.
\end{coro}

Now, Lemma \ref{lem:simplify} shows that studying $(\bfk[x],P)$-modules is equivalent to studying modules of the Jordan plane $\mathcal{A}_{y_{2}}$ in the usual sense. Thus, descriptions of the irreducible $\mathcal{A}_{y^{2}}$-module can be interpreted in terms of the $(\bfk[x],P)$-module.

\begin{defn}
A nonzero $(\bfk[x],P)$-module $(M,p)$ is called irreducible if the submodule of $(M,p)$ is either $(0,p)$ or $(M,p)$. $(M,p)$ is called indecomposable if $M \neq 0$ and $(M,p)$ is not the direct sum of its two proper submodules.
\end{defn}

As a special case of $\mathcal{A}_{h}$, the following result can be obtained from the result for irreducible $\mathcal{A}_{y^{2}}$-modules over a field of characteristic zero (see Theorem 4.7 and Corollary 6.1 in \cite{BLO2}).

\begin{theorem}{\rm (\cite{BLO2})}
A $(\bfk[x],P)$-module $(M,p)$ is irreducible if and only if $M$ is an irreducible $\bfk[x]$-module and $p=0$.
\end{theorem}

A semisimple $(\bfk[x],P)$-module is a direct sum of irreducible $(\bfk[x],P)$-modules, so we have the following corollary.
\begin{coro}
Let $(M,p)$ be a $(\bfk[x],P)$-module. Then, $(M,p)$ is semisimple if and only if $M$ is a semisimple $\bfk[x]$-module and $p=0$.
\end{coro}

\begin{remark}
Let $M$ be an indecomposable $\bfk[x]$-module. If there is some $p \in \E(M)$ such that $(M,p)$ is a $(\bfk[x],P)$-module, then $(M,p)$ is indecomposable. We find all such $p$ for a given indecomposable $\bfk[x]$-module in Section \ref{sec:classify}. In addition, it is natural to ask whether all the indecomposable $(\bfk[x],P)$-modules are derived from indecomposable $\bfk[x]$-modules with some suitable $p$ and the answer is no. For example, let $M=\bfk v_{1} \oplus \bfk v_{2}$ with
$$
x(v_{1},v_{2}) = (v_{1},v_{2})\left(\begin{array}{cc}a&0\\0&a \end{array}\right), \quad p(v_{1},v_{2}) = (v_{1},v_{2})\left(\begin{array}{cc}0&1\\0&0 \end{array}\right).
$$
The only $1$-dimensional $(\bfk[x],P)$-module is $(\bfk,0)$, so $(M,p)$ is an indecomposable $(\bfk[x],P)$-module whereas $M$ is a decomposable $\bfk[x]$-module.
\end{remark}

\subsection{Direct decomposition of $(M,p)$} \label{decom}
In this subsection, we generalize the results concerning the direct decomposability of $(\bfk[x],P)$-modules in \cite{Iy} from algebraically closed fields of characteristic zero to fields of characteristic zero.

It should be noted that our method is slightly different from that employed by Iyudu.
Any $\mathcal{A}_{y^{2}}$-module can be regarded as both a $\bfk[x]$-module and a $\bfk[y]$-module, but the role of action $x$ is different from the role of action $y$. For example, Corollary \ref{thm:sim} states that the action $y$ should be nilpotent. Thus, there are two natural ways to study $\mathcal{A}_{y^{2}}$-modules: trying to recover the $\mathcal{A}_{y^{2}}$-module structures on $\bfk[x]$-modules or suitable $\bfk[y]$-modules. Iyudu's method aims to determine the $\mathcal{A}_{y^{2}}$-module structures on a given $\bfk[y]$-module with the nilpotent action $y$. However, our starting point is the representation of $(\bfk[x],P)$, so it is natural for us to determine the $(\bfk[x],P)$-module structures on a given $\bfk[x]$-module.

For a $\bfk[x]$-module $M$, let $\tau(v)=xv, \forall v \in M$, and thus $\tau \in {\rm End}_{\bfk}(M)$. It is well known that a $\bfk[x]$-module $M$ can be regarded as a $\bfk$-vector space $M$ with a $\bfk$-linear map $\tau \in {\rm End}_{\bfk}(M)$ and $f(x)v=f(\tau)v$, $f(x) \in \bfk[x]$. In the following, the linear map induced by the action of $x$ is always denoted by $\tau$. Therefore, a $(\bfk[x],P)$-module $(M,p)$ can be regarded as a $\bfk$-vector space $M$ with two $\bfk$-linear maps, $\tau$ and $p$.

Let $(M,p)$ be a $(\bfk[x],P)$-module and the minimal polynomial of $\tau$
$$
f(x) = h_{1}^{r_{1}}(x)\cdots h_{s}^{r_{s}}(x), \quad r_{1}, \cdots ,r_{s} \geq 1,
$$
where $h_{1}(x),\cdots, h_{s}(x)$ are distinct irreducible polynomials in $\bfk[x]$. Define
$$
M_{h_{i}^{r_{i}}(x)} = \{v \in M~|~h_{i}^{t}(\tau)v=0, \mbox{for some}~t>0\}, \quad 1 \leq i \leq s.
$$
Then, $M_{h_{i}^{r_{i}}(x)}$ is invariant under $\tau$ and
\begin{equation}\label{equ:generalized}
M = M_{h_{1}^{r_{1}}(x)} \oplus \cdots \oplus M_{h_{s}^{r_{s}}(x)}.
\end{equation}

\begin{lemma}\label{lem:lie}
Let $M$ be a $\bfk[x]$-module and $p \in \E(M)$. If $(\ad \tau)^{\ell}p=0, \ell \geq 1$, then $M_{h_{i}^{r_{i}}(x)}$ is invariant under $p$.
\end{lemma}
\begin{proof}
We prove this statement by induction on $\ell$.
When $\ell=1$, we find that $\tau$ commutes with $p$. For any $v \in M_{h_{i}^{r_{i}}(x)}$, we have some $t>0$ such that $h_{i}^{t}(\tau)v=0$. $p$ commutes with $h_{i}(\tau)^{t}$, so we obtain
$$
h_{i}^{t}(\tau)pv = p h_{i}^{t}(\tau)v=0.
$$
Hence, $p(v) \in M_{h_{i}^{r_{i}}(x)}$, as required.

We assume that this statement is true for $\ell \geq 1$ and we consider the case where $\ell+1$. Now, $(\ad \tau)^{\ell+1}p=0$ implies that $(\ad \tau)^{\ell}[\tau,p]=0$. By the induction hypothesis, we find that $M_{h_{i}^{r_{i}}(x)}$ is invariant under $[\tau,p]$. For any $v \in M_{h_{i}^{r_{i}}(x)}$, there is some $t>0$ such that $h_{i}^{t}(\tau)v=0$. Using $h_{i}^{t}(\tau)p$ acting on $v$, we obtain
\begin{eqnarray*}
h_{i}^{t}(\tau)p v &=& p h_{i}^{t}(\tau) v + [h_{i}^{t}(\tau),p] v \\
&=& \sum_{j=0}^{t-1} h_{i}^{t-j-1}(\tau)[\tau,p]h_{i}^{j}(\tau) v .
 \end{eqnarray*}
Since $M_{h_{i}^{r_{i}}(x)}$ is invariant under $h_{i}^{t-j-1}(\tau)$, $[\tau,p]$, and $h_{i}^{j}(\tau)$, then we have $h_{i}^{t}(\tau)p v \in M_{h_{i}^{r_{i}}(x)}$. Some $t'$ exists such that
$$
0=h_{i}^{t'}(\tau) (h_{i}^{t}(\tau) p v) =h_{i}^{t'+t}(\tau) p v.
$$
Therefore, $p(v) \in M_{h_{i}^{r_{i}}(x)}$. This completes the induction.
\end{proof}

\begin{lemma}\label{lem:inva}
Let $(M,p)$ be a $(\bfk[x],P)$-module. Then, $M_{h_{i}^{r_{i}}(x)}$ is invariant under $p$.
\end{lemma}
\begin{proof}
By Lemma \ref{lem:ker}, we have
$$
\tau p^{s}-p^{s}\tau = sp^{s+1},s\geq1.
$$
Then, $(\ad \tau)^{s} p = s! p^{s+1}$.
By Corollary \ref{thm:nil}, $p$ is nilpotent and we let the index of $p$ be $\ell+1$. If $p=0$, then the statement is obviously true. Now, we assume that $p\neq0$, and thus $\ell \geq 1$. Now, we have $(\ad \tau)^{\ell} p = \ell! p^{\ell+1}=0$. By Lemma \ref{lem:lie}, we find that $M_{h_{i}^{r_{i}}(x)}$ is invariant under $p$.
\end{proof}

\begin{theorem}\label{thm:decom}
Let $(M,p)$ be a $(\bfk[x],P)$-module with
$$
M = M_{h_{1}^{r_{1}}(x)} \oplus \cdots \oplus M_{h_{s}^{r_{s}}(x)}.
$$
Then, $p_{i}=p|_{M_{h_{i}^{r_{i}}(x)}}$ is a linear map on $M_{h_{i}^{r_{i}}(x)}$, $(M_{h_{i}^{r_{i}}(x)},p_{i})$ is a $(\bfk[x],P)$-submodule of $(M,p)$, and
\begin{equation}\label{equ:decom}
(M,p) = \bigoplus_{i=1}^{s} (M_{h_{i}^{r_{i}}(x)},p_{i}).
\end{equation}
\end{theorem}
\begin{proof}
The proof follows immediately from Lemma \ref{lem:inva}.
\end{proof}

Theorem \ref{thm:decom} states that the problem of finding the structure of $(\bfk[x],P)$ module $(M,p)$ is reduced to the problem of finding the structure of its submodule $(M_{h_{i}^{r_{i}}(x)},p_{i})$.

\begin{prop}
Let $(M,p)$ be a $(\bfk[x],P)$-module. If $\tau$ has $n$ distinct eigenvalues, then $p=0$.
\end{prop}
\begin{proof}
Now, $(M,p)$ is the direct sum of an $n$ 1-dimensional $(\bfk[x],P)$-module. By Corollary \ref{coro:onedim}, we obtain $p=0$.
\end{proof}

If $\bfk$ is an algebraically closed field of characteristic zero, then Eq. (\ref{equ:generalized}) becomes the generalized eigenspace decomposition of $M$
\begin{equation}\label{equ:aclosed}
M = M_{(x-a_{1})^{r_{1}}} \oplus \cdots M_{(x-a_{s})^{r_{s}}},
\end{equation}
where $a_{1},\cdots, a_{s}$ are all distinct eigenvalues of $\tau$ and $r_{1}+\cdots+r_{s}=n$.

\begin{coro}\label{cor:mblock} {\rm (\cite{Iy})}
Let $\bfk$ be an algebraically closed field of characteristic zero and $(M,p)$ a $(\bfk[x],P)$-module. Then, $p_{i}=p|_{M_{(x-a_{i})^{r_{i}}}}$ is a linear map on $M_{(x-a_{i})^{r_{i}}}$ and
\begin{equation}\label{equ:acdecom}
(M,p) = \bigoplus_{i=1}^{s} (M_{(x-a_{i})^{r_{i}}},p_{i}).
\end{equation}
\end{coro}

\begin{coro}\label{cor:block}{\rm (\cite{Iy})}
Let $\bfk$ be an algebraically closed field of characteristic zero, $(M,p)$ a $(\bfk[x],P)$-module,and $A, B$ the matrices of $\tau,p$ with respect to the basis $v_{1},\cdots,v_{n}$ respectively. If
$$
A =\left(\begin{array}{cccc}A_{11}&0&\cdots&0 \\
0&A_{22}&\cdots&0\\
\cdots&\cdots&\cdots&\cdots\\
0&0&0&A_{ss} \end{array} \right)
$$ is a quasi-diagonal matrix with $A_{11},\cdots,A_{ss}$ have no common eigenvalues, and thus
$$
B =\left(\begin{array}{cccc}B_{11}&0&\cdots&0 \\
0&B_{22}&\cdots&0\\
\cdots&\cdots&\cdots&\cdots\\
0&0&0&B_{ss} \end{array} \right)
$$ is also a quasi-diagonal matrix with $A_{ii}B_{ii}-B_{ii}A_{ii} = B_{ii}^{2}, i=1,\cdots,s$.
\end{coro}

\begin{prop}\label{prop:gen}
Let $\bfk$ be an algebraically closed field of characteristic zero and $(M,p)$ a $(\bfk[x],P)$-module. Some $\bfk$-basis $v_{1},\cdots,v_{n}$ of $M$ exists such that the matrices associated with $\tau, p$ are both upper triangular.
\end{prop}

\begin{proof}
Let the index of $p$ be $\ell$ and $L=\spa \{\tau,p,\cdots,p^{\ell-1}\}$.
By Lemma \ref{lem:ker}, we have
\begin{equation*}
\tau p^{s}-p^{s}\tau = sp^{s+1}, \quad \forall s \geq 1.
\end{equation*}
Let $[,]$ be the commutator and  we have
$$
[\tau,\tau]=0, \quad [\tau,p^{s}] = sp^{s+1}, \quad [p^{s},p^{t}]=0,1\leq s,t \leq \ell-1.
$$
Then, $L$ is a finite dimensional Lie algebra under $[,]$. The derived algebra is
$$
[L,L] = \spa \{p^{2}, \cdots, p^{\ell-1}\}.
$$
Hence, $L$ is solvable. By Lie's theorem, a basis exists such that $\tau$ and $p$ are both upper triangular.
\end{proof}

\begin{coro}
Let $(M,p)$ be a $(\bfk[x],P)$-module. A basis $v_{1}, \cdots, v_{n}$ of $M$ exists such that $(V_{i},p_{i})$ is a $(\bfk[x],P)$-module, where $V_{i}=\spa\{v_{1},\cdots, v_{i}\}$, $p_{i}=p|_{V_{i}}$, $1 \leq i \leq n$.
\end{coro}

\begin{remark}
When $\bfk$ is an algebraically closed field of characteristic zero, the problem of determining the structure of the $(\bfk[x],P)$-modules is reduced to the case where $\tau$ has exactly one eigenvalue. Furthermore, by Eq. {\rm(\ref{equ:main})} in Corollary \ref{thm:sim}, it is possible to assume that $\tau$ has exactly one eigenvalue $0$.
\end{remark}

\section{Rota--Baxter module structure on $\mathcal{M}_{a,n}$}\label{sec:classify}
In this section, we mainly focus on the Rota--Baxter module structures on the indecomposable $\bfk[x]$-modules. After introducing some basic properties of Rota--Baxter module structures on $\mathcal{M}_{a,n}$, we classify them up to isomorphism. We conclude this section by discussing the general $(\bfk[x],P)$-modules. In this section,
we assume that $\bfk$ is an algebraically closed field of characteristic zero.

\subsection{Matrices associated with $(M,p)$} Let $M$ be a $\bfk[x]$-module, $p:M\longrightarrow M$ a $\bfk$-linear map, and fix a $\bfk$-basis $v_{1}, v_{2}, \cdots, v_{n}$ of $M$. Let the matrices of $\tau$ and $p$ corresponding to the basis $v_{1},v_{2},\cdots,v_{n}$ be $A$ and $B$, respectively. Then, $(M,p)$ is a $(\bfk[x],P)$-module if and only if
$$
AB-BA=B^{2}.
$$

It is well known that there is a suitable $\bfk$-basis $v_{1},\cdots, v_{n}$ of $M$ such that $\tau$ has canonical Jordan block form and each Jordan block is a matrix of the following form.
$$
J_{r}(a) = \left(\begin{array}{ccccc}
a&1&\cdots&0&0 \\
0&a&\cdots&0&0 \\
\vdots&\vdots&\cdots&\vdots&\vdots \\
0&0&\cdots&a&1 \\
0&0&\cdots&0&a
\end{array} \right).
$$

By Corollary \ref{cor:block}, we only need to consider the case where $\tau$ has exactly one eigenvalue. In general, $\tau$ may have several Jordan blocks. Thus, we consider the simplest case where $\tau$ has exactly one Jordan block. Now, the $(\bfk[x],P)$ module $(M,p)$ is $(\mathcal{M}_{a,n},p)$, where $\mathcal{M}_{a,n} = \bfk[x]/(x-a)^{n}$ with the standard basis $(\overline{x-a)^{n-1}},\cdots, \overline{x-a},\overline{1}$. It is easy to see that
\begin{equation*}
J_{n}(a)B-BJ_{n}(a) = B^{2}
\end{equation*}
is equivalent to
\begin{equation*}
J_{n}(0)B-BJ_{n}(0) = B^{2}.
\end{equation*}
Thus, it is possible to assume that the module is $(\mathcal{M}_{n},p):=(\mathcal{M}_{n,0},p)$. Denote $J_{n}(0)$ by $J_{n}$ when there is no possibility of confusion.

\begin{prop}\label{prop:upper}
Let $(\mathcal{M}_{n},p)$ be a $(\bfk[x],P)$-module. Then, $p(\overline{x^{n-1}})=0$.
\end{prop}

\begin{proof}
We prove the statement by induction on $n$. When $n=1$, $p=0$, as necessary. Assume that the statement holds for $n-1 \geq 0$ and consider the case $n$.

$\tau(\overline{x^{n-1}}) = \overline{x x^{n-1}} =\overline{x^{n}} = 0$, so we find that $\overline{x^{n-1}}$ is an eigenvector of $\tau$ corresponding to $0$. In addition, $\rank(\tau)=n-1$ implies that the dimension of the eigenspace $V$ associated with $0$ is $1$. Thus,$V ={\spa} \{\overline{ x^{n-1} }\}$.

By Lemma \ref{lem:ker}, we have
$$
\tau p^{s}-p^{s}\tau=sp^{s+1}, \forall s \geq 1.
$$

If $p^{s+1}(\overline{x^{n-1}})=0$, then we have
\begin{eqnarray*}
\tau(p^{s}(\overline{x^{n-1}})) = (\tau p^{s})(\overline{x^{n-1}})=(\tau p^{s}-p^{s}\tau)(\overline{x^{n-1}}) = sp^{s+1}(\overline{x^{n-1}})=0.
\end{eqnarray*}
Then, $p^{s}(\overline{x^{n-1}}) \in V$ and $p^{s}(\overline{x^{n-1}})=\mu \overline{x^{n-1}}$ for some $\mu\in \bfk$, i.e., $\overline{x^{n-1}}$ is an eigenvector of $p^{s}$ corresponding to $\mu$. $p$ is nilpotent, so $\mu=0$ and thus $p^{s}(\overline{x^{n-1}})=0$. Hence, $p^{s+1}(\overline{x^{n-1}})=0$ implies that $p^{s}(\overline{x^{n-1}})=0$. By Theorem \ref{thm:nil}, $p$ is nilpotent and some $m$ exists such that $p^{m+1} =0$.
It is obvious that $p^{m+1}(\overline{x^{n-1}})=0$. Therefore, we have $p(\overline{x^{n-1}})=0$.
\end{proof}


\begin{coro}\label{coro:quo}
Let $(\mathcal{M}_{n},p)$ be a $(\bfk[x],P)$-module. Then, $p$ induces a linear map $p'$ on $\mathcal{M}_{n-1} = \mathcal{M}_{n}/(\overline{x^{n-1}})$ given by $p'(v+(\overline{x^{n-1}}))=p(v)+(\overline{x^{n-1}})$ such that $(\mathcal{M}_{n-1},p')$ is also a $(\bfk[x],P)$-module.
\end{coro}
\begin{proof}
Proposition \ref{prop:upper} implies that $p$ induces a linear map $p'$ on $\mathcal{M}_{n-1}=\mathcal{M}_{n}/(\overline{x^{n-1}})$. For any $v+(\overline{x^{n-1}})\in \mathcal{M}_{n-1}$, we have
\begin{eqnarray*}
(\tau p-p\tau)(v+(\overline{x^{n-1}}))&=& \tau p(v+(\overline{x^{n-1}}))-p\tau(v+(\overline{x^{n-1}})) \\
&=& (\tau p(v)-p\tau(v))+(\overline{x^{n-1}})= p^{2}(v)+(\overline{x^{n-1}})\\
&=&(p')^{2}(v+(\overline{x^{n-1}})).
\end{eqnarray*}
 Therefore, $(\mathcal{M}_{n-1},p')$ is a $(\bfk[x],P)$-module.
\end{proof}

\noindent{\bf Notation.}~~Let $(\mathcal{M}_{n},p)$ be a $(\bfk[x],P)$-module. In the rest of this section, $B=(b_{ij})_{n\times n}\in \Mnk$ always denotes the matrix of $p$ corresponding to the standard basis, unless specified otherwise.

\begin{coro}\label{coro:upper}
Let $(\mathcal{M}_{n},p)$ be a $(\bfk[x],P)$-module. Then, $B$ is strictly upper triangular.
\end{coro}
\begin{proof}
Using induction on $n$, when $n=1$, $B=0$, as necessary. The induction step follows from Corollary \ref{coro:quo}.
\end{proof}

\begin{prop}\label{prop:expli}
Let $(\mathcal{M}_{n},p)$ be a $(\bfk[x],P)$-module. Then, $B$ satisfies
\begin{eqnarray}
&b_{ij}=0, \quad 1 \leq j \leq i \leq n& \label{equ:expli1}\\
&b_{i+1,j+1}-b_{i,j} = \displaystyle \sum_{k=i+1}^{j} b_{i,k}b_{k,(j+1)}, 1 \leq i < j \leq n-1. \label{equ:expli2}&
\end{eqnarray}
\end{prop}

\begin{proof}
By Corollary \ref{coro:upper}, $B$ is strictly upper triangular. After directly computing and then comparing the two sides of the matrix equation
$$
J_{n}B-BJ_{n}= B^{2},
$$
we have Eq. (\ref{equ:expli1}) and Eq. (\ref{equ:expli2}).
\end{proof}

\begin{lemma}\label{lem:der}
For each $1\leq \ell \leq n-2$, $b_{i,i+\ell}$, $1 \leq i \leq n-\ell-1$ are
completely determined recursively by $b_{n-1,n},b_{n-2,n},\cdots,b_{n-\ell,n}$. In particular, $B$ is completely determined by its $n$-th column.
\end{lemma}
\begin{proof}
We prove this lemma by induction on $\ell$. First, we rearrange Eq. (\ref{equ:expli2}) as follows:
\begin{eqnarray}
\ell=1, \qquad b_{i+1,i+2} -b_{i,i+1}&=&b_{i,i+1}b_{i+1,i+2}, \quad  1\leq i \leq n-2; \label{equ:eq1}\\
\cdots \qquad \qquad  \qquad \qquad &\cdots&  \qquad \qquad \qquad \qquad  \cdots \notag \\
\ell = \ell, \qquad b_{i+1,i+\ell+1} -b_{i,i+\ell}&=& \sum_{k=i+1}^{i+\ell}b_{i,k}b_{k,i+\ell+1}, \quad  1\leq i \leq n-\ell-1;\label{equ:eq2}\\
\cdots \qquad \qquad  \qquad \qquad &\cdots&   \qquad \qquad \qquad \qquad \qquad  \cdots \notag \\
\ell = n-2, \qquad b_{2,n} -b_{1,n-1}&=& \sum_{k=2}^{n-1}b_{1,k}b_{k,n}, \quad   i=1. \label{equ:eq3} \qquad \qquad
\end{eqnarray}
When $\ell=1$, by Eq. (\ref{equ:eq1}), we obtain
\begin{equation}\label{equ:for}
b_{i,i+1} = \frac{b_{i+1,i+2}}{1+b_{i+1,i+2}}, \quad  1 \leq i \leq n-2.
\end{equation}
Then, $b_{i,i+1}$, $1\leq i \leq n-2$ are determined by $b_{n-1,n}$, and the formulae are
$$
b_{12} = \frac{b_{n-1,n}}{1+(n-2)b_{n-1,n}},\quad b_{23}=\frac{b_{n-1,n}}{1+(n-3)b_{n-1,n}}, \quad \cdots, \quad b_{n-2,n-1}=\frac{b_{n-1,n}}{1+b_{n-1,n}}.
$$
Assume that the statement is true for $1 \leq \ell \leq n-3$ and we consider the case for $\ell+1$. By Eq. (\ref{equ:eq2}), we obtain
\begin{equation}\label{equ:step2}
b_{i,i+\ell+1} = \frac{b_{i+1,i+\ell+2}(1-b_{i,i+1}) -(\sum_{k=i+2}^{i+\ell}b_{i,k}b_{k,i+\ell+2})}{1+b_{i+\ell+1,i+\ell+2}},1\leq i \leq n-\ell-1.
\end{equation}
By the induction hypothesis, we find that $b_{i,i+1}, b_{i+\ell+1,i+\ell+2}, b_{i,k}b_{k,i+\ell+2}, k=i+1,\cdots,i+\ell$ are all determined by $b_{n-1,n},b_{n-2,n},\cdots,b_{n-\ell,n}$. By combining this with Eq. (\ref{equ:step2}), we find that $b_{i,i+\ell+1}$, $1\leq i \leq n-\ell-1$ are determined by $b_{n-1,n},b_{n-2,n},\cdots,b_{n-\ell,n}$ and $b_{n-\ell-1,n}$, which completes the induction.
\end{proof}

\begin{coro}
Let $(\mathcal{M}_{n},p)$ be a $(\bfk[x],P)$-module. Then, $p$ is determined completely by $p(\overline{1})$.
\end{coro}
\begin{proof}
Given that $(b_{1n},\cdots, b_{n-1,n},0)^{T}$ are the coordinates of $p(\overline{1})$ under the basis $\overline{x^{n-1}},\cdots, \overline{x} , \overline{1}$, then Lemma \ref{lem:der} implies that $p$ is determined completely by $p(\overline{1})$.
\end{proof}

\begin{prop}\label{prop:zero}
Let $(\mathcal{M}_{n},p)$ be a $(\bfk[x],P)$-module. If $B$ satisfies $b_{n-1,n}=\cdots=b_{n-i+1,n}=0$ and $1 \leq t-s \leq i-1$, then $b_{s,t}=0$.
\end{prop}
\begin{proof}
We prove the statement by induction on $\ell=t-s$. When $\ell=1$, we have $b_{12}=\cdots=b_{n-1,n}=0$. Assume that the statement is true for $1\leq \ell \leq i-2$ and consider the case of $\ell+1$. By Proposition \ref{prop:expli}, we have
$$
b_{s+1,s+\ell+2}-b_{s,s+\ell+1} = \displaystyle \sum_{k=s+1}^{s+\ell} b_{s,k}b_{k,(s+\ell+2)}.
$$
Since $s+1\leq k \leq s+\ell$, i.e., $k-s\leq \ell$, then we have $b_{s,k}=0$ and thus $b_{1,\ell+2}=\cdots = b_{n-\ell-1,n}=0$, which completes the induction.
\end{proof}

\begin{prop}\label{prop:center}
If $n \geq 3, b_{n-1,n}=\cdots=b_{n-i+1,n}=0$, $b_{n-i,n} \neq 0$, $[\frac{n-1}{2}] < i \leq n-1$, then
$$
B = b_{n-i,n}(J_{n})^{i}+ b_{n-i-1,n}(J_{n})^{i+2}+\cdots+b_{1,n}(J_{n})^{n-1}.
$$
\end{prop}

\begin{proof}
By Proposition \ref{prop:expli}, we have
$$
b_{s+1,t+1}-b_{s,t} = \displaystyle \sum_{k=s+1}^{t} b_{s,k}b_{k,(t+1)}, 1 \leq s < t \leq n-1.
$$
By Proposition \ref{prop:zero}, either $k-s \leq i-1$ or $t+1-k \leq i-1$ implies that $b_{s,k}b_{k,(t+1)}=0$. If $b_{s+1,t+1}-b_{s,t} \neq 0$, and there is at least one $s+1 \leq k_{0} \leq t$ such that $k_{0}-s \geq i$ and $t+1-k_{0} \geq i$. Then,
$$
k_{0}-s+t+1-k_{0}=t-s+1 \geq 2i.
$$
Thus, $t+1 \geq 2i+s>n-1+s>n$, which contradicts $t \leq n-1$. Hence, $b_{s+1,t+1}=b_{s,t}, 1 \leq s < t \leq n-1$ and thus
$$
B = b_{n-i,n}(J_{n})^{i}+ b_{n-i-1,n}(J_{n})^{i+1}+\cdots+b_{1,n}(J_{n})^{n-1}.
$$
\end{proof}

\begin{prop}\label{prop:mid}
If $n\geq 5, b_{n-1,n}=\cdots=b_{n-i+1,n}=0$, $b_{n-i,n} \neq 0$, $2\leq i \leq [\frac{n-1}{2}]$, then
$$
b_{n-t,n}=b_{n-t-1,n-1}=\cdots = b_{1,t+1}, \quad i \leq t \leq 2i-2.
$$
\end{prop}
\begin{proof}
Since
\begin{eqnarray*}
&b_{n-t-j,n-j}-b_{n-t-j-1,n-j-1} = \displaystyle \sum_{k=n-t-j}^{n-j-1}b_{n-t-j-1,k}b_{k,n-j},~~0\leq j \leq n-t-1,&\\
&k-(n-t-j-1)+(n-j)-k=t+1 \leq 2i-2+1=2i-1,&
\end{eqnarray*}
then we have either $k-(n-t-j-1)<i$ or $(n-j)-k<i$. Hence,
$$
b_{n-t-j,n-j}-b_{n-t-j-1,n-j-1}=0,~~0\leq j \leq n-t-1.
$$
\end{proof}

\begin{exam}\label{exam:ex}
\begin{enumerate}
\item If $b_{n-1,n}=b\neq 0, b_{n-2,n}=c$, then
$$
B = \left(\begin{array}{ccccccc} 0&\frac{b}{1+(n-2)b}& c-(n-3)b^{2}&\cdots&\ast&\ast&b_{1n} \\
0&0&\frac{b}{1+(n-3)b}&\cdots&\ast&\ast&b_{2n}  \\
\vdots&\vdots&\vdots&\ddots&\vdots&\vdots&\vdots\\
0&0&0&\cdots&\frac{b}{1+2b}&c-b^{2}&b_{n-3,n}\\
0&0&0&\cdots&0&\frac{b}{1+b}&c\\
0&0&0&\cdots&0&0&b\\
0&0&0&\cdots&0&0&0
\end{array}\right).
$$
\item If $n=7$, $b_{67}=b_{57}=b_{47}=0$ and $b_{37}\neq0$, i.e., $i=4>[\frac{7-1}{2}]$, then
$$
B=\left(\begin{array}{ccccccc} 0&0&0&0&b_{37}&b_{27}&b_{17} \\
0&0&0&0&0&b_{37}&b_{27} \\
0&0&0&0&0&0&b_{37}\\
0&0&0&0&0&0&0\\
0&0&0&0&0&0&0\\
0&0&0&0&0&0&0\\
0&0&0&0&0&0&0
\end{array}\right).
$$
\item If $n=7$, $b_{67}=b_{57}=0$ and $b_{47}\neq0$, i.e., $i=3 \leq [\frac{7-1}{2}]$, then
$3 \leq t \leq 4$ and
$$
B=\left(\begin{array}{ccccccc} 0&0&0&b_{47}&b_{37}&b_{27}-b_{47}^{2}&b_{17} \\
0&0&0&0&b_{47}&b_{37}&b_{27} \\
0&0&0&0&0&b_{47}&b_{37}\\
0&0&0&0&0&0&b_{47}\\
0&0&0&0&0&0&0\\
0&0&0&0&0&0&0\\
0&0&0&0&0&0&0
\end{array}\right).
$$
\end{enumerate}
\end{exam}

\subsection{Classification of $(\mathcal{M}_{n},p)$}\label{subsec:class}
Recall that an $(\mathcal{M}_{n},p)$ can be regarded as $\mathcal{M}_{n}$ with two linear maps $\tau,p$ and the matrices of $\tau, p$ are $J_{n},B$ under the basis $\overline{x^{n-1}},\cdots, \overline{x},\overline{1}$, respectively.
The stabilizer of $J_{n}$ under the action of conjugation is
\begin{eqnarray*}
G_{n} &=& \{S \in GL(n,\bfk)~|~J_{n}=S^{-1}J_{n}S\}\\
&=& \{S~|~S=s_{0}I+s_{1}J_{n}+\cdots+s_{n-1}J_{n}^{n-1}, s_{0}\neq 0, s_{1},\cdots,s_{n-1} \in \bfk\},
\end{eqnarray*}
and the centralizer of $G_{n}$ is
\begin{eqnarray*}
Z(G_{n})= \{X \in {\Mnk}~|~X=a_{0}I+a_{1}J_{n}+\cdots+a_{n-1}J_{n}^{n-1}, a_{0},\cdots,a_{n-1} \in \bfk\}.
\end{eqnarray*}

Let $(\mathcal{M}_{n},p)$, $(\mathcal{M}_{n},q)$ be two isomorphic $(\bfk[x],P)$-modules and the matrices of $p,q$ corresponding to the standard basis are $B_{p}, B_{q}$ respectively. Then, $(\mathcal{M}_{n},p)$ is isomorphic to $(\mathcal{M}_{n},q)$ if and only if  a $S\in G_{n}$ exists such that $B_{p}=S^{-1}B_{q}S$. Let $X_{j}$ be the $j$-th column (from left to right) of $X \in {\Mnk}$ and $X_{j}(i)$ is the $i$-th array of $X_{j}$. It is easy to see that $B_{n}$ are the coordinates of $p(\overline{1})$ under the standard basis and $B_{n}(i)$ is the coefficient of $\overline{x^{n-i}}$.

Let $\calp(\mathcal{M}_{n})= \{p \in {\rm End_{\bfk}}(\mathcal{M}_{n})~|~ (\mathcal{M}_{n},p)~ \mbox{is a}~ (\bfk[x],P)\mbox{-module}\}$ and define a map $\Psi$ from $\calp(\mathcal{M}_{n})$ to $\bfk^{n}$ by
\begin{eqnarray*}
\Psi: \calp(\mathcal{M}_{n}) &\longrightarrow& \bfk^{n}, \\
p &\mapsto& B_{n},
\end{eqnarray*}
i.e., $\Psi(p) =B_{n}= (b_{1n},\cdots,b_{n-1,n},0)^{T}$.

\begin{defn}
Let $n \geq 2$, $p \in \calp(\mathcal{M}_{n})$ and $\Psi(p) =(b_{1n},\cdots,b_{n-1,n},0)^{T}$. If $b_{n-1,n}=\cdots=b_{n-i+1,n}=0$ and $b_{n-i,n}\neq 0$, then the integer $1\leq i\leq n-1$ is called {\bf the depth} of $p$ or $B$ and it is denoted by $\dep(p)$ or $\dep(B)$, respectively. If $p=0$, we define $\dep(p)=\dep(B)=n$. By Corollary \ref{coro:upper}, we have $1\leq \dep(B) \leq n$.
\end{defn}

\begin{lemma}\label{lem:pre1}
Let $(\mathcal{M}_{n},p)$ be a $(\bfk[x],P)$-module.
\begin{enumerate}
\item[(1)] If $\dep(B)=1$, then
\begin{equation}\label{equ:inv}
(S^{-1}BS)_{n}(j)=B_{n}(j),\quad \forall~ S \in G_{n},~~ j=n-1,n.
\end{equation}
\item[(2)] If $\dep(B)=i$, $2\leq i \leq [\frac{n-1}{2}]$, then
\begin{equation}\label{equ:inv2}
(S^{-1}BS)_{n}(j)=B_{n}(j), \quad \forall~ S \in G_{n}, j=n-2i+1,\cdots,n.
\end{equation}
\end{enumerate}
\end{lemma}
\begin{proof}
\noindent (1) ~~ We may assume that $S=I+s_{1}J_{n}+\cdots+s_{n-1}J_{n}^{n-1}$ and $S^{-1}=I+t_{1}J_{n}+\cdots+t_{n-1}J_{n}^{n-1}$. We have
\begin{eqnarray*}
(S^{-1}BS)_{n} = (S^{-1}BS)\varepsilon_{n} = S^{-1}B \left(\begin{array}{c}s_{n-1}\\
\vdots \\ s_{1} \\1 \end{array}\right) = \begin{pmat}({....})
&\ast&\ast& \cr \-
0&\cdots&1&t_{1}\cr
0&\cdots&0&1\cr
\end{pmat}  \left(\begin{array}{c}\ast\\
\vdots \\ \ast \\ b_{n-1,n} \\0 \end{array}\right) = \left(\begin{array}{c}\ast\\
\vdots \\ \ast \\ b_{n-1,n} \\0 \end{array}\right) .
\end{eqnarray*}
Then, $(S^{-1}BS)_{n}(j)=B_{n}(j), \quad j=n-1,n$.

\noindent (2)~~By Proposition \ref{prop:mid} and if we let $\star = b_{n-i,n}J_{n}^{i}+ \cdots + b_{n-2i+2,n}J_{n}^{2i-2}$, then we have
$$
B = \star+ (B-\star).
$$
Since $\star \in Z(G_{n})$, then we have
$$
S^{-1}BS= \star+ S^{-1}(B-\star)S,
$$
\begin{eqnarray*}
S^{-1}(B-\star)S &=& S^{-1}\left(\begin{array}{ccccccc}
0&\cdots&\ast&\ast&\cdots&\ast&b_{1n} \\
0& \cdots&0&\ast&\cdots&\ast&b_{2n}  \\
0& \cdots&\vdots&\vdots&\ddots&\vdots&\vdots\\
0& \cdots&0&0&\cdots&\ast&b_{n-2i,n}\\
0& \cdots&0&0&\cdots&0&b_{n-2i+1,n}\\
0& \cdots&0&0&\cdots&0&0\\
0& \cdots&\vdots&\vdots&\cdots&\vdots&\vdots\\
0& \cdots&0&0&\cdots&0&0
\end{array}\right) \left(\begin{array}{c}s_{n-1}\\
\vdots \\ s_{1} \\1 \end{array}\right) \\
&=& \begin{pmat}({....})
&&\ast& \ast &&\ast\cr \-
0&\cdots&0&1&\cdots&t_{2i-1}\cr
\vdots&&\vdots&&\ddots\cr
0&\cdots&0&0&\cdots&1\cr
\end{pmat} \left(\begin{array}{c}\ast \\ \vdots\\\ast\\ b_{n-2i+1,n}\\0
\\ \vdots \\0\end{array}\right)=\left(\begin{array}{c}\ast \\ \vdots\\\ast\\ b_{n-2i+1,n}\\0
\\ \vdots \\0\end{array}\right).
\end{eqnarray*}
Then, $(S^{-1}BS)_{n}(j)=B_{n}(j), \quad j=n-2i+1,\cdots,n.$
\end{proof}

\begin{lemma}\label{lem:orbits}
Let $(\mathcal{M}_{n},p)$ be a $(\bfk[x],P)$-module and $n \geq 2$. If $\dep(B)=1$ and $b_{n-1,n} =b\neq 0$, then $b \notin \{0,-1,\cdots, -\frac{1}{n-2}\}$ and some $S \in G_{n}$ exists such that $(S^{-1}BS)_{n}=(0,\cdots,0,b,0)^{T}$, i.e.,
$$
S^{-1}BS = \mathcal{B} = \left(\begin{array}{ccccccc} 0&\frac{b}{1+(n-2)b}&0&\cdots&0&0&0 \\
0&0&\frac{b}{1+(n-3)b}&\cdots&0&0&0  \\
\vdots&\vdots&\vdots&\ddots&\vdots&\vdots&\vdots\\
0&0&0&\cdots&\frac{b}{1+2b}&0&0\\
0&0&0&\cdots&0&\frac{b}{1+b}&0\\
0&0&0&\cdots&0&0&b\\
0&0&0&\cdots&0&0&0
\end{array}\right).
$$
\end{lemma}
\begin{proof}
The first statement follows immediately from Proposition \ref{prop:expli}. We prove the second statement by induction on $n \geq 2$. The case where $n=2$ is obvious. Assume that the statement holds for any $n-1 \geq 2$ and consider the case for $n$. $B$ is strictly upper triangular and $J_{n}B-BJ_{n}=B^{2}$, so we have
$$
B = \begin{pmat}({.|...})
0&&&\ast&\ast \cr\-
0&&&& \cr
\vdots&&&\widehat{B}& \cr
0& & &&\cr
\end{pmat},
$$
where $\widehat{B} \in {\rm M}^{n-1}(\bfk)$ and thus $J_{n-1}\widehat{B}-\widehat{B}J_{n-1}=\widehat{B}^{2}$. By the induction hypothesis, some $\widehat{S} \in G_{n-1}$ exists such that
$$
\widehat{S}^{-1}\widehat{B}~\widehat{S} = \widehat{\mathcal{B}} = \left(\begin{array}{cccccc}
0&\frac{b}{1+(n-3)b}&\cdots&0&0&0  \\
\vdots&\vdots&\ddots&\vdots&\vdots&\vdots\\
0&0&\cdots&\frac{b}{1+2b}&0&0\\
0&0&\cdots&0&\frac{b}{1+b}&0\\
0&0&\cdots&0&0&b\\
0&0&\cdots&0&0&0
\end{array}\right),
$$
and $\widehat{S} = I_{n-1}+s_{1}J_{n-1}+\cdots + s_{n-2}(J_{n-1})^{n-2}$ for some $s_{1}, s_{2},\cdots, s_{n-2} \in \bfk$. Let
$$
T=  \begin{pmat}({.|...})
1&&s_{1}&\cdots&s_{n-1} \cr\-
0&&&& \cr
\vdots&&&\widehat{S}& \cr
0& & &&\cr
\end{pmat} = \left(\begin{array}{cccccc} 1&s_{1}&\cdots&s_{n-3}&s_{n-2}&s_{n-1} \\
0&1&\cdots&s_{n-4}&s_{n-3}&s_{n-2} \\
\vdots&\vdots&\ddots&\vdots&\vdots&\vdots\\
0&0&\cdots&1&s_{1}&s_{2}\\
0&0&\cdots&0&1&s_{1}\\
0&0&\cdots&0&0&1
\end{array}\right).
$$
Then,
$$
T^{-1} B T =  \begin{pmat}({.|...})
1&&\ast&\cdots&\ast \cr\-
0&&&& \cr
\vdots&&&\widehat{S}^{-1}& \cr
0& & &&\cr
\end{pmat}\begin{pmat}({.|...})
0&&\ast&\cdots&\ast \cr\-
0&&&& \cr
\vdots&&&\widehat{B}& \cr
0& & &&\cr
\end{pmat}\begin{pmat}({.|...})
1&&\ast&\cdots&\ast \cr\-
0&&&& \cr
\vdots&&&\tilde{S}& \cr
0& & &&\cr
\end{pmat} = \begin{pmat}({.|...})
0&&x_{1}&\cdots&x_{n-1} \cr\-
0&&&& \cr
\vdots&&&\widehat{S}^{-1}\widehat{B}~\widehat{S}& \cr
0& & &&\cr
\end{pmat},
$$
for some $x_{1}, \cdots, x_{n-1} \in \bfk$.
$T^{-1}BT$ still satisfies
 \begin{equation*}
J_{n}(T^{-1}BT)-(T^{-1}BT)J_{n}  = (T^{-1}BT)^{2},
\end{equation*}
so we have
$$
\frac{b}{1+(n-3)b} -x_{1} = \frac{b}{1+(n-3)b}x_{1},
$$
$$
\left(\frac{b}{1+(n-i-2)b}+1\right)x_{i}=0, \quad  2 \leq i \leq n-2.
$$

Then, $x_{1}=\frac{b}{1+(n-2)b}$ and $x_{2}=\cdots =x_{n-2}=0$. Let
$$
Q = \left(\begin{array}{ccccc} 1&0&\cdots&\alpha&0 \\
0&1&\cdots&0&\alpha \\
\vdots&\vdots&\ddots&\vdots&\vdots\\
0&0&\cdots&1&0\\
0&0&\cdots&0&1
\end{array}\right) \in G_{n}, \quad Q^{-1} = \left(\begin{array}{ccccc} 1&0&\cdots&-\alpha&0 \\
0&1&\cdots&0&-\alpha \\
\vdots&\vdots&\ddots&\vdots&\vdots\\
0&0&\cdots&1&0\\
0&0&\cdots&0&1
\end{array}\right),
$$
where $\alpha= \frac{(1+(n-2)b)x_{n-1}}{(n-2)b^{2}}$, and then we have
$$
Q^{-1} (T^{-1}BT) Q = \mathcal{B}.
$$
If we let $S=TQ$,then we obtain the result required.
\end{proof}

\begin{lemma}\label{lem:mid1}
Let $(\mathcal{M}_{n},p)$ be a $(\bfk[x],P)$-module and $n \geq 5$. If $\dep(B)=i$ and $2 \leq i \leq [\frac{n-1}{2}]$, then some $S \in G_{n}$ exists such that $
\mathcal{B}_{n}=(S^{-1}BS)_{n}=(0,\cdots,0,b_{n-2i+1,n},\cdots,b_{n-i,n},0,\cdots,0)^{T}.$
Let $b_{n-i,n}=b \neq 0$ and $b_{n-2i+1,n}=c$, and then
$$
S^{-1}BS = \mathcal{B} = b_{n-i,n}J_{n}^{i}+ \cdots + b_{n-2i+2,n}J_{n}^{2i-2} + \left(\begin{array}{ccccccc}
0&\cdots&c-(n-2i)b^{2}&\ast&\cdots&\ast&0 \\
0& \cdots&0&\ast&\cdots&\ast&0  \\
0& \cdots&\vdots&\vdots&\ddots&\vdots&\vdots\\
0& \cdots&0&0&\cdots&c-b^{2}&0\\
0& \cdots&0&0&\cdots&0&c\\
0& \cdots&0&0&\cdots&0&0\\
0& \cdots&\vdots&\vdots&\cdots&\vdots&\vdots\\
0& \cdots&0&0&\cdots&0&0
\end{array}\right).
$$
\end{lemma}

\begin{proof}
We prove the statement by induction on $n \geq 5$. When $n=5$ and now $i=2$, we have
$$
B=\left(\begin{array}{ccccc}0&0&b_{35}&b_{25}-b_{35}^{2}&b_{15}\\ 0&0&0&b_{35}&b_{25}\\0&0&0&0&b_{35}\\
0&0&0&0&0\\
0&0&0&0&0\end{array}\right).
$$
Let
$$
S= I + \alpha J_{n}^{n-2i} = I +\alpha J_{5} = \left(\begin{array}{ccccc}1&\alpha&0&0&0\\ 0&1&\alpha&0&0\\0&0&1&\alpha&0\\
0&0&0&1&\alpha\\
0&0&0&0&1\end{array}\right)
$$
and then
$$
(S^{-1}BS)_{5}(1)= \alpha (b_{25}-b_{35}^{2})+b_{15}-\alpha b_{25}-\alpha^{2}b_{35}+\alpha^{2}b_{35}=b_{15}-b_{35}^{2}\alpha.
$$
Let $\alpha =\frac{b_{15}}{b_{35}^{2}}$. By Lemma \ref{lem:pre1}, we have
$$
(S^{-1}BS)_{5} = (0,b_{25},b_{35},0,0)^{T}
$$
as required.

Assume that the statement holds for any $n-1 \geq 5$ and consider the case for $n$. $B$ is strictly upper triangular and $J_{n}B-BJ_{n}=B^{2}$, so we have
$$
B = \begin{pmat}({.|...})
0&&\ast&\cdots&\ast \cr\-
0&&&& \cr
\vdots&&&\widehat{B}& \cr
0& & &&\cr
\end{pmat},
$$
where $\widehat{B} \in {\rm M}^{n-1}(\bfk)$ and thus $J_{n-1}\widehat{B}-\widehat{B}J_{n-1}=\widehat{B}^{2}$. By the induction hypothesis, some $\widehat{S} \in G_{n-1}$ exists such that
$$
\widehat{S}^{-1}\widehat{B}~\widehat{S} = \widehat{\mathcal{B}},
$$
and $\widehat{S} = I_{n-1}+s_{1}J_{n-1}+\cdots + s_{n-2}(J_{n-1})^{n-2}$ for some $s_{1}, s_{2},\cdots, s_{n-2} \in \bfk$. Let
$$
T=  \begin{pmat}({.|...})
1&&s_{1}&\cdots&s_{n-1} \cr\-
0&&&& \cr
\vdots&&&\widehat{S}& \cr
0& & &&\cr
\end{pmat} = \left(\begin{array}{cccccc} 1&s_{1}&\cdots&s_{n-3}&s_{n-2}&s_{n-1} \\
0&1&\cdots&s_{n-4}&s_{n-3}&s_{n-2} \\
\vdots&\vdots&\ddots&\vdots&\vdots&\vdots\\
0&0&\cdots&1&s_{1}&s_{2}\\
0&0&\cdots&0&1&s_{1}\\
0&0&\cdots&0&0&1
\end{array}\right),
$$
and we have
\begin{eqnarray*}
T^{-1} BT &=&  \begin{pmat}({.|...})
1&&\ast&\cdots&\ast \cr\-
0&&&& \cr
\vdots&&&\widehat{S}^{-1}& \cr
0& & &&\cr
\end{pmat}\begin{pmat}({.|...})
0&&\ast&\cdots&\ast \cr\-
0&&&& \cr
\vdots&&&\widehat{B}& \cr
0& & &&\cr
\end{pmat}\begin{pmat}({.|...})
1&&\ast&\cdots&\ast \cr\-
0&&&& \cr
\vdots&&&\widehat{S}& \cr
0& & &&\cr
\end{pmat} \\
&=& \begin{pmat}({.|...})
0&& \xi \cr\-
0_{(n-1)\times 1}&&\widehat{S}^{-1}\widehat{B}~\widehat{S} \cr
\end{pmat},
\end{eqnarray*}
where $\xi = (\ast,\cdots,\ast, x_{1,n})_{1 \times(n-1)} \in \bfk^{n-1}$.
By Proposition \ref{prop:mid} and if we let $\star = b_{n-i,n}J_{n}^{i}+ \cdots + b_{n-2i+2,n}J_{n}^{2i-2} \in Z(G_{n})$, then we have
$$
T^{-1}BT = \star+ (T^{-1}BT-\star).
$$
From the induction hypothesis and
\begin{equation*}
J_{n}(T^{-1}BT)-(T^{-1}BT)J_{n}  = (T^{-1}BT)^{2},
\end{equation*}
it follows that
$$
(T^{-1}BT-\star) = \left(\begin{array}{ccccccc}
0&\cdots&b_{1,2i}&\ast&\cdots&\ast&x_{1n} \\
0& \cdots&0&\ast&\cdots&\ast&0  \\
0& \cdots&\vdots&\vdots&\ddots&\vdots&\vdots\\
0& \cdots&0&0&\cdots&\ast&0\\
0& \cdots&0&0&\cdots&0&b_{n-2i+1,n}\\
0& \cdots&0&0&\cdots&0&0\\
0& \cdots&\vdots&\vdots&\cdots&\vdots&\vdots\\
0& \cdots&0&0&\cdots&0&0
\end{array}\right).
$$
Let
$$
Q = I_{n} + \alpha (J_{n})^{n-2i}  \in G_{n}.
$$
Then,
$$
Q^{-1}(T^{-1} B T)Q = \star + Q^{-1}(T^{-1}BT-\star)Q,
$$
\begin{eqnarray*}
(Q^{-1}(T^{-1} B T)Q)_{n}(1) &=& (Q^{-1}(T^{-1}BT-\star)Q)_{n}(1)\\
&=& \alpha (b_{1,2i}-b_{n-2i+1,n})+x_{1n}.
\end{eqnarray*}
In addition, by Eq. (\ref{equ:eq2}), we have
\begin{eqnarray*}
b_{1,i+1}=\cdots = b_{n-i,n}, \quad b_{s+1,s+2i}-b_{s,s+2i-1}=b_{s,s+i}b_{s+i,s+2i}, \quad 1\leq s \leq n-2i.
\end{eqnarray*}
Then,
$$
b_{n-2i+1,n}-b_{1,2i}= (n-2i)b_{n-i,n}^{2}.
$$
Let $\alpha = \frac{x_{1n}}{(n-2i)b_{n-i,n}^{2}}$ and $S=TQ$. Then,
$$
(S^{-1}BS)_{n} =(0,\cdots,0,b_{n-2i+1,n},\cdots,b_{n-i,n},0,\cdots,0)^{T},
$$
which completes the induction.
\end{proof}

\begin{defn}
{\rm Let $n \geq 2$ and a $(\bfk[x],P)$-module $(\mathcal{M}_{n},p)$ is called canonical if $p$ satisfies one of the following conditions:
\begin{enumerate}
\item $p=0$;
\item $\dep(p)=1$ and $\Psi(p) =(0,0,\cdots,0,b_{n-1,n},0)^{T}$;
\item $2\leq\dep(p)=i \leq [\frac{n-1}{2}]$ and $\Psi(p) =(0,\cdots,0, b_{n-2i+1,n},\cdots,b_{n-i,n},0,\cdots,0)^{T}$;
\item $\max\{[\frac{n-1}{2}],1\}< \dep(p)=i \leq n-1$ and $\Psi(p) =(b_{1n},\cdots,b_{n-i,n},0,\cdots,0)^{T}$.
\end{enumerate}
}
\end{defn}

\begin{theorem}
There is only one $1$-dimensional $(\bfk[x],P)$-module $(\mathcal{M}_{1},0)$. For $n \geq 2$, each $(\bfk[x],P)$-module $(\mathcal{M}_{n},p)$ is isomorphic to a canonical $(\bfk[x],P)$-module.
\end{theorem}
\begin{proof}
The statement follows immediately from Corollary \ref{coro:onedim}, Proposition \ref{prop:center}, Lemma \ref{lem:pre1}, Lemma \ref{lem:orbits}, and Lemma \ref{lem:mid1}.
\end{proof}

\begin{exam}
We now list some low-dimensional canonical $(\bfk[x],P)$-modules: for the given $n \geq 2$, let $b \notin \{0,-1,\cdots, -\frac{1}{n-2} \}$ and $c \neq 0$.
\begin{enumerate}
\item $n=2$: $\Psi(p)=\left(\begin{array}{c} b\\0 \end{array}\right)$; $p=0$.
\item $n=3$: $\Psi(p)=\left(\begin{array}{c} 0\\b\\0 \end{array}\right)$; $\Psi(p)=\left(\begin{array}{c} c\\0 \\0 \end{array}\right)$; $p=0$.
\item $n=4$: $\Psi(p)=\left(\begin{array}{c} 0\\0 \\b\\0 \end{array}\right)$; $\Psi(p)=\left(\begin{array}{c} d\\c \\0\\0 \end{array}\right)$; $\Psi(p)=\left(\begin{array}{c} c\\0 \\0\\0 \end{array}\right)$; $p=0$.
\item $n=5$: $\Psi(p)=\left(\begin{array}{c} 0\\0 \\0\\b \\0 \end{array}\right)$; $\Psi(p)=\left(\begin{array}{c} 0\\d \\c\\0 \\0 \end{array}\right)$; $\Psi(p)=\left(\begin{array}{c} d\\c \\0\\0 \\0 \end{array}\right)$; $\Psi(p)=\left(\begin{array}{c} c\\0 \\0\\0 \\0 \end{array}\right)$; $p=0$.
\end{enumerate}
\end{exam}

\subsection{Discussion of the general cases}
Let $(M,p)$ be an $n$-dimensional $(\bfk[x],P)$-module. Now, we assume that $\tau$ has exactly one eigenvalue $0$ and
$$
M=\mathcal{M}_{r_{1}}\oplus \cdots \oplus \mathcal{M}_{r_{s}}, \quad r_{1} +\cdots +r_{s}=n.
$$
We characterized the $(\bfk[x],P)$-modules $(M_{r_{j}},p_{j})$, $1 \leq j \leq s$ in Subsection \ref{subsec:class}. If we let $p=\sum_{j=1}^{s}p_{j}$, then $(M,p)$ is an $n$-dimensional $(\bfk[x],P)$-module. However, these are not the only $n$-dimensional $(\bfk[x],P)$-modules. In fact, Proposition \ref{prop:gen} implies that some basis $v_{1},\cdots ,v_{n}$ of $M$ exists such that each $(V_{i},p_{i})$, $i =1, \cdots, n$, is an $i$-dimensional $(\bfk[x],P)$-module. In general, unfortunately, if we write $M=\mathcal{M}_{r_{1}}\oplus \cdots \oplus \mathcal{M}_{r_{s}}$, then we do not have
\begin{eqnarray*}
p(\mathcal{M}_{r_{j}}) &\subseteq& \mathcal{M}_{r_{j}},\\
p(\mathcal{M}_{r_{1}} \oplus \cdots \oplus \mathcal{M}_{r_{j}}) &\subseteq & \mathcal{M}_{r_{1}} \oplus \cdots \oplus \mathcal{M}_{r_{j}}, \quad 1 \leq j \leq s.
\end{eqnarray*}
For example, let
$M=\mathcal{M}_{1}\oplus \mathcal{M}_{2}$ and then $A=\diag\{J_{1},J_{2}\}$. Let the matrix of $p$ be
$$
B=\left(\begin{array}{ccc} 1&0&1 \\ 0&0&1 \\ -1&0&-1 \end{array}\right)
$$
and then $(M,p)$ is a $(\bfk[x],P)$-module. Obviously, $B$ is not quasi-diagonal and not quasi-upper triangular. Classifying $(M,p)$ is generally  complicated. We conclude this section only by giving the classification of $(M,p)$ when $s=n$.

\begin{prop}\label{prop:squ}
Let $M=\mathcal{M}_{r_{1}}\oplus \cdots \oplus \mathcal{M}_{r_{s}}$ be an $n$-dimensional $\bfk[x]$-module and $p \in {\rm End}_{\bfk}(M)$. If $s=n$, then $(M,p)$ is a $(\bfk[x],P)$-module if and only if $p^{2}=0$.
\end{prop}
\begin{proof}
$M=\mathcal{M}_{1}^{\oplus n} \cong \bfk^{\oplus n}$,  so $\tau$ is the zero map. By Theorem \ref{thm:sim}, $(M,p)$ is a $(\bfk[x],P)$-module if and only if
$$
0p-p0=p^{2},
$$
i.e., $p^{2}=0$.
\end{proof}

\begin{coro}
Let $M_{2}=\bfk \oplus \bfk$ and $p_{2}$ be the $\bfk$-linear map defined by $p_{2}(1_{\bfk},0)=(0,0), p_{2}(0,1_{\bfk})=(1_{\bfk},0)$. Then, $(M_{2},p_{2})$ is a $(\bfk[x],P)$-module.
\end{coro}

\begin{theorem}
Let $M=\bfk^{\oplus n}$. If $(M,p)$ is a $(\bfk[x],P)$-module, then some $0 \leq \ell \leq [\frac{n}{2}]$ exists such that
\begin{equation*}
M \cong (k,0)^{\oplus (n-2\ell)} \bigoplus (M_{2},p_{2})^{\oplus \ell}.
\end{equation*}
\end{theorem}
\begin{proof}
By Proposition \ref{prop:squ}, we have $p^{2}=0$. ${\rm Aut}_{\bfk}(M)={\rm Aut}_{\bfk}(\bfk^{\oplus n}) = {\rm gl}(n,\bfk)$ and some $\sigma \in {\rm gl}(n,\bfk)$ exists such that $\sigma^{-1}p\sigma$ has the canonical Jordan block form, and  $(\sigma^{-1}p\sigma)^{2}=0$. Thus, $p$ has exactly one eigenvalue $0$ and the sizes of the Jordan blocks of $p$ are $1$ or $2$. If $p$ has $\ell$ Jordan blocks of size $2$, then $p$ has $n-2\ell$ Jordan blocks of size $1$. Therefore, we obtain
\begin{equation*}
M \cong (k,0)^{\oplus (n-2\ell)} \bigoplus (M_{2},p_{2})^{\oplus \ell}.
\end{equation*}
\end{proof}

\noindent {\bf Acknowledgements:} The authors thank Professor Zongzhu Lin for his encouragement. This study was supported by the NNSF of China (Grant No. 11501466 ), FRF for the Central Universities XDJK2015C143, and the Ph.D. startup research foundation of Southwest University SWU114096. The authors thank
the anonymous referees for helpful comments.

\end{document}